\documentclass[12pt,a4paper,oneside,reqno,notitlepage]{amsart}
\usepackage{amssymb}

\usepackage{amsfonts}
\usepackage{amscd,amsmath}
\usepackage{euscript}
\usepackage{comment}

\usepackage[arrow,matrix,curve]{xy}\SilentMatrices
\def\xyma{\xymatrix@M.7em}

\newcommand{\para}{\par\vspace{.25cm}}

\numberwithin{equation}{section}

\usepackage[T2A]{fontenc}
\newtheorem{cor}{Corollary}[section]

\newtheorem{prop}{Proposition}[section]
\newtheorem{theorem}{Theorem}[section]
\newtheorem{lemma}{Lemma}[section]

\newtheorem{example}{Example}[section]

\def\Z{{\mathbb{Z}}}

\def\bee{\begin{equation}}
\def\ee{\end{equation}}

\def\br{\mathbf{r}}

\begin{document}
\title{Symmetric Ideals in group rings  and  Simplicial Homotopy}
\author{Roman Mikhailov, Inder Bir S. Passi and Jie Wu} \thanks{The research of the first author was partially supported by
Dynasty Foundation, Russian Presidental Grant MK-3644.2009.01 and
RFBR-08-01-00663-a. The research of the third author was partially supported by the Academic Research Fund of the
National University of Singapore R-146-000-101-112.}

\maketitle

\begin{abstract}
In this paper homotopical methods for the  description
of subgroups determined by ideals in group rings are introduced. It is shown that in
certain cases the subgroups determined by symmetric
product of ideals in group rings can be described with the help of
homotopy groups of spheres.
\end{abstract}

\section{Introduction}

The purpose of this paper is to use the homotopical methods for
the description of subgroups determined by certain ideals, here called symmetric ideals,  in free
group rings.
\para
Let $F$ be a free group, $\mathbb Z[F]$ its integral group
ring and $I$ a two-sided ideal in $\mathbb Z[F]$. The general
problem of description of the normal subgroup
$$
D(F;I):=F\cap (1+I)
$$
of $F$ is very difficult. As an illustration of the complexity of answers for different particular cases we may mention  some  examples.  Let $R$
be a normal subgroup of $F$, ${\bf r}=(R-1)\mathbb Z[F]$, and ${\bf
f}$ the augmentation ideal of $\mathbb Z[F],$ then  \cite{CKG}
\begin{align*}
& F\cap (1+{\bf f}^2{\bf r}^2)=\gamma_3(R\cap [F,F])\gamma_4(R),\\
& F\cap (1+{\bf rf}^2{\bf r})=[R\cap [F,F],R\cap
[F,F],R]\gamma_4(R).
\end{align*} The subtility of the dimension subgroup problem is well-known; this is the case when $I={\bf f}^n+{\bf r}$.
For a survey of problems in this area, see \cite{Passi}, \cite{Gupta}, \cite{MP}.\para

Given a ring $A$ and two-sided ideals $I_1,\dots, I_n\ (n\geq 2)$
in $A$ consider their symmetric product:
$$
(I_1\dots I_n)_S:=\sum_{\sigma \in \Sigma_n}I_{\sigma_1}\dots
I_{\sigma_n},
$$
where $\Sigma_n$ is the  symmetric group of degree $n$. For example, in the
case $n=2$, one has $(I_1I_2)_S=I_1I_2+I_2I_1.$ Observe that while
 $(I_1\dots I_n)_S\subseteq I_1\cap
\dots \cap I_n$ always, the reverse inclusion does not hold, in general.\para  Let $F$ be a free group, and let
$R_1,\dots, R_n$ be normal subgroups of $F$. Consider the induced two-sided ideals in the integral group ring $\mathbb Z[F]$ defined as
${\bf r}_i=(R_i-1)\mathbb Z[F],\
i=1,\dots, n$. The following problems arise naturally:
\begin{quote}
(1) Identify the quotient
$$
Q({\bf r}_1,\dots, {\bf r}_n):=\frac{{\bf r}_1\cap \dots \cap {\bf
r}_n}{({\bf r}_1\dots {\bf r}_n)_S}.
$$
(2) Identify the normal subgroup of $F$, determined by the ideal $({\bf
r_1}\dots {\bf r}_n)_S,$  i.e., the subgroup
$$
D(F; ({\bf r}_1\dots {\bf r}_n)_S):=F\cap (1+({\bf r}_1\dots {\bf
r}_n)_S).
$$\end{quote}
Let $[R_1,\,\dots\,,\,R_n]_S$ denote the symmetric commutator subgroup, namely,\\  $ \prod_{\sigma\in \Sigma_n}[\dots\,
[R_{\sigma_1},\,R_{\sigma_2}],\,\dots\,,\, R_{\sigma_n}]$ of the  normal subgroups $R_1,\,\dots\,,\, R_n$:
$$
[R_1,\,\dots\,,\,R_n]_S=\prod_{\sigma\in \Sigma_n}[\dots
[R_{\sigma_1},\,R_{\sigma_2}],\,\dots\,,\, R_{\sigma_n}].
$$
Observe that we have always $$ [R_1,\,\dots\,,\, R_n]_S\subseteq D(F;({\bf r}_1\,\dots\,
{\bf r}_n)_S).
$$The fundamental theorem of free group rings (see \cite{Gupta}, Theorem 3.1, p.\,12) states that, for all $n\geq 1$, the above inclusion is an equality in case $R_i=F$ for $i=1,\,\ldots\,,\,n$. Apart from this case, there is hardly anything else that seems to be available in the literature about the subgroups $D(F;({\bf r}_1\,\dots\,
{\bf r}_n)_S)$, in general. Naturally, one would like to investigate the validity, or otherwise, of the inclusion $D(F;({\bf r}_1\,\dots\,
{\bf r}_n)_S)\subseteq [R_1,\,\dots\,,\,R_n]_S$.
Let
\begin{align*}
& f_{F;R_1,\,\dots\,,\,R_n}: \frac{R_1\cap \dots \cap
R_n}{[R_1,\,\dots\,,\,R_n]_S}\to \frac{{\bf r}_1\cap \dots \cap
{\bf r}_n}{({\bf r}_1\dots {\bf r}_n)_S},\end{align*} be the
natural map defined by \begin{align*} & f_{F;R_1,\,\dots\,,\,R_n}:
g.[R_1,\,\dots\,,\,R_n]_S\mapsto g-1+({\bf r}_1\dots {\bf
r}_n)_S,\ g\in R_1\cap\dots\cap R_n.
\end{align*}
The main idea of this paper is based on the fact that, for a
certain choice of groups $F,\,R_1,\,\dots\,,\, R_n$, there exists
a space $X$, such that the map $f_{F;R_1,\,\dots\,,\, R_n}$ is the
$(n-1)$st Hurewicz homomorphism:
$$
\xyma{ \frac{R_1\cap\, \dots\, \cap R_n}{[R_1,\,\dots\,,\,R_n]_S} \ar@{=}[d]
\ar@{->}[rr]^{\ \ \ \ \ \ \ \ \ \  \ \ \ f_{F;R_1,\,\dots\,,\, R_n}\ \ \
\ \ \ \ \ \ \ \ \ } & & \frac{{\bf r}_1\cap
\dots \cap {\bf r}_n}{({\bf r}_1\,\dots \,{\bf r}_n)_S}\ar@{=}[d]\\
\pi_{n-1}(X) \ar@{->}[rr] & & H_{n-1}(X) }
$$
In that case, the quotient $\frac{D(F;({\bf r}_1\,\dots\, {\bf
r}_n)_S)}{[R_1,\,\dots\,,\, R_n]_S}$ is exactly the kernel of Hurewicz
homomorphism and we are able to use arguments from simplicial homotopy  for
the computation of subgroups determined by symmetric product of
ideals. Our analysis also yields an example where the inclusion
 $$D(F;({\bf r}_1\,\dots\,
{\bf r}_n)_S)\supseteq [R_1,\,\dots\,,\,R_n]_S$$ is proper.

In Section 2 we prove certain technical results needed for our investigation. Our main results are Theorems 3.1 and 3.2 (see Section 3).
\par\vspace{1cm}
\section{Technical results} We need some preparation for proving our main results. Given a group $G$, let $\Delta(G)$ denote the augmentation ideal of its integral group ring $\mathbb Z[G]$. The following result is well-known.
\para
\begin{lemma}\label{element}
If $N$ is a normal subgroup of a group $G$, then $N\cap
(1+\Delta(N)\Delta(G))=[N,\,N].$
\end{lemma}
\para
For the case of two normal subgroups in the free group $F$,  we have the following\para
\begin{prop}
Let $F=R_1R_2$. Then the map
$$
f_{F;\,R_1,\,R_2}: \frac{R_1\cap R_2}{[R_1,\,R_2]}\to \frac{{\bf
r}_1\cap {\bf r}_2}{{\bf r}_1{\bf r}_2+{\bf r_2}{\bf r}_1}
$$
is an isomorphism. In particular, $D(F;\, ({\bf r}_1{\bf r}_2)_S)=[R_1,\,R_2]$.
\end{prop}
\para\noindent{\it Proof.}
Let $T=\{t_i\}_{i\in I}\subseteq R_1$ be a left transversal for
$R_2$ in $F$. Then every element $f\in F$ can be written uniquely
as $f=ts$ with $t\in T$ and $s\in R_2$; in particular, if $f\in
R_1$, then $s\in R_1\cap R_2$. Let $\varphi:\mathbb Z[F]\to
\mathbb Z[R_2]$ be the $\mathbb Z$-linear extension of the map
$F\to R_2$ given by $f\mapsto s$. Observe that ${\bf r}_1 {\bf
r}_2=\Delta(R_1)\Delta(R_2)$ and ${\bf r}_2{\bf
r}_1=\Delta(R_2)\Delta(R_1)$ since $F=R_1R_2$. Furthermore,
$$
\varphi({\bf r}_1{\bf r}_2+{\bf r}_2{\bf r}_1)\subseteq
\Delta(R_1\cap R_2)\Delta(R_2)+\Delta([R_1,\,R_2])\mathbb Z[R_2].
$$

\par\vspace{.25cm}
Consider the map $$\theta:R_1\cap R_2\to \frac{{\bf r}_1\cap {\bf
r}_2}{{\bf r}_1{\bf r}_2+{\bf r}_2{\bf r}_1},\quad f\mapsto
f-1+{\bf r}_1{\bf r}_2+{\bf r}_2{\bf r}_1.$$ Clearly $\theta $ is
a homomorphism and $[R_1,\,R_2]\subseteq \ker \theta$. Let $f\in
R_1\cap R_2$ be an element in $\ker \theta$. We then have
$$f-1=\varphi(f-1)\in \Delta(R_1\cap R_2)\Delta(R_2)+\Delta([R_1,\,R_2])\mathbb Z[R_2]$$ in the group ring $\mathbb Z[R_2]$.
Thus, going modulo $[R_1,\,R_2]$ and invoking Lemma \ref{element}
with $G=R_1/[R_1,\,R_2],\ N=(R_1\cap R_2)/[R_1,\,R_2],$ we must
have $f\in [R_1,\,R_2]$. Consequently $\theta$ induces
 a monomorphism
$$f_{R_1,R_2}:\frac{R_1\cap R_2}{[R_1,\,R_2]}\hookrightarrow \frac{{\bf r}_1\cap {\bf
r}_2}{{\bf r}_1{\bf r}_2+{\bf r}_2{\bf r}_1}.$$
\par\vspace{.25cm}
Let $\alpha\in {\bf r}_1$. Then $\alpha=\sum_i(r_i-1)\beta_i$ with
$r_i\in R_1$ and $\beta_i\in \mathbb Z[R_2]$. Now
$r_i=t_{i_j}s_{i_J}$ with $t_{i_j}\in T$ and $s_{i_j}\in R_1\cap
R_2$. Therefore,
$$\alpha\equiv (w-1)+\sum_km_k(t_k-1)\ \mod\  {\bf r}_1{\bf r}_2+{\bf r}_2{\bf r}_1$$ with $m_k\in \mathbb Z$ and $w\in R_1\cap
R_2$. It follows that if $\alpha\in {\bf r}_1\cap {\bf r}_2$, then
$m_k=0$ for all $k$, and we thus conclude that $f_{R_1,R_2}$ is an
epimorphism and hence an isomorphism. $\Box$

\para

\begin{lemma}\label{techlemma} Let $X=X_1\sqcup \dots \sqcup X_n\ (n\geq 2)$ be a disjoint union of
sets.  Let \linebreak $p_i: F(X)\to F(X_1\sqcup \ldots \hat{X}_i\ldots\sqcup X_n),\ i=1,\ldots,n$, be the natural projections induced by $$p_i(x)=
\begin{cases}
x, \ for\ x\in X\backslash X_i\\
1,\ for\ x\in X_i
                                                                                                                                     \end{cases}$$
and $R_i=\ker (p_i). $
 Then
\begin{quote}

$(i)$ $ R_1\cap \dots
\cap R_n=[R_1,\dots, R_n]_S
$
in $F(X)$;\para\noindent
$ (ii)$ ${\bf r}_1\cap \dots\cap {\bf r}_k=({\bf
r}_1\dots {\bf r}_n)_S
$
in $\mathbb Z[F(X)]$.\end{quote}
\end{lemma}

\begin{proof}
The statement  (i) follows from (\cite{Wu}, Corollary 3.5) (see   \cite{BMVW}).
\para
For the proof of (ii) observe first that, for each $i,j\in
\{1,2,\dots, n\},\ i\neq j,$ we have
$$
R_i=\langle X_i\rangle^{F(X\setminus X_j)}[R_i,\,R_j].
$$
Since $\mathbb Z[F(X)]=\mathbb Z[F(X\setminus X_j)]+{\bf r}_j$ and
$[R_i,\,R_j]-1\subseteq {\bf r}_i{\bf r}_j+{\bf r}_j{\bf r}_i,$ it
follows that
\begin{equation}\label{02}
{\bf r}_i=(\langle X_i\rangle^{F(X\setminus X_j)}-1)\mathbb
Z[F(X\setminus X_j)]+{\bf r}_i{\bf r}_j+{\bf r}_j{\bf r}_i,
\end{equation}
and consequently, we have
$$
{\bf r}_i\cap {\bf r}_j={\bf r}_i{\bf r}_j+{\bf r}_j{\bf r}_i,\
i\neq j.
$$
Suppose that, for some $k$, $2\leq k<n,$ we have shown that
$$
{\bf r}_{i1}\cap \dots \cap {\bf r}_{ik}=({\bf r}_{i1}\dots {\bf
r}_{ik})_S
$$
for all subsets of $k$ elements from $\{1,\dots,n\},$ and let $j$
be an integer, $1\leq j\leq n,\ j\notin \{i1,\dots, ik\}.$ From
(\ref{02}), we have
$$
{\bf r}_{il}=(\langle X_{il}\rangle^{F(X\setminus X_j)}-1)\mathbb
Z[F(X\setminus X_j)]+{\bf r}_{il}{\bf r}_j+{\bf r}_j{\bf r}_{il},\
l=i1,\dots, ik.
$$
Consequently
$$
({\bf r}_{i1}\dots {\bf r}_{ik})_S\subseteq \mathbb Z[F(X\setminus
X_j)]+({\bf r}_{i1}\dots {\bf r}_{ik} {\bf r}_j)_S.
$$
An application of the natural projection $\mathbb Z[F(X)]\to
\mathbb Z[F(X\setminus X_j)]$ induced by the map which is
identity on $X\setminus X_j$ and trivial on $X_j$ then shows that
$$
{\bf r}_j\cap {\bf r}_{i1}\cap \dots \cap {\bf r}_{ik}\subseteq
({\bf r}_{i1}\dots {\bf r}_{ik}{\bf r}_j)_S.
$$
The reverse inclusion being trivial, it follows that the
intersection of $k+1$ distinct ideals out of ${\bf r}_1,\dots,
{\bf r}_n$ equals  the corresponding symmetric sum of their
products, and thus, by induction, assertion (ii) is proved.
\end{proof}

\par\vspace{.5cm}
\section{Simplicial constructions}
\subsection{Milnor's construction.}
Recall that, for a given pointed simplicial set $K$, the Milnor
$F(K)$-construction \cite{Milnor} is the simplicial group with
$F(K)_n=F(K_n\setminus *)$, where $F(-)$ is the free group
functor. Consider the simplicial circle
$S^1=\Delta[1]/\partial\Delta[1]$:
\begin{equation}\label{1-sphere}
S_0^1=\{*\},\ S_1^1=\{*,\,\sigma\},\ S_2^1=\{*,\, s_0\sigma,\,
s_1\sigma\},\,\dots\,,\, S_n^1=\{*,\, x_0,\,\dots\,,\, x_{n-1}\},
\end{equation}
where $x_i=s_{n-1}\dots \hat s_i\dots s_0\sigma$. For the Milnor
construction $F(S^1)$, $F(S^1)_n$ is a free group of rank $n$,
for $n\geq 1:$
$$F(S^1):\ \ \ \ldots\ \begin{matrix}\longrightarrow\\[-3.5mm] \ldots\\[-2.5mm]\longrightarrow\\[-3.5mm]
\longleftarrow\\[-3.5mm]\ldots\\[-2.5mm]\longleftarrow \end{matrix}\ F_3\ \begin{matrix}\longrightarrow\\[-3.5mm]\longrightarrow\\[-3.5mm]\longrightarrow\\[-3.5mm]\longrightarrow\\[-3.5mm]\longleftarrow\\[-3.5mm]
\longleftarrow\\[-3.5mm]\longleftarrow
\end{matrix}\ F_2\ \begin{matrix}\longrightarrow\\[-3.5mm] \longrightarrow\\[-3.5mm]\longrightarrow\\[-3.5mm]
\longleftarrow\\[-3.5mm]\longleftarrow \end{matrix}\ \mathbb Z$$
with face and degeneracy homomorphisms:
\begin{align*}
& \partial_i: F_n\to F_{n-1},\ i=0,\,\dots\,,\, n,\ n=2,\,3,\,\dots\\
& s_i: F_n\to F_{n+1},\ i=0,\,\dots\,,\, n,\ n=1,\,2,\,\dots\,.
\end{align*}
There is a homotopy equivalence \cite{Milnor}:
$$
|F(S^1)|\simeq \Omega S^2.
$$
Hence, for $n\geq 2$, the $n$th homotopy group of $S^2$ can be
described as an intersection of kernels in degree $n-1$ modulo
simplicial boundaries. Following \cite{Wu}, denote the elements
from a basis of $F_{n+1}$ as follows:
\begin{align*}
& y_n=s_{n-1}\dots s_0\sigma,\\
& y_i=s_n\dots \hat s_i\dots s_0\sigma (s_n\dots \hat s_{i+1}\dots
s_0\sigma)^{-1},\ 0\leq i<n.
\end{align*}
Then, it follows from standard simplicial identities that, in the
free group $F_{n+1}$, one has
\begin{align*}
& \ker(\partial_0)=y_0\dots y_n,\\ & \ker(\partial_{i})=\langle
y_{i-1}\rangle^{F_{n+1}},\ 0<i\leq n+1
\end{align*}
Lemma \ref{techlemma} applied to the case $X=\{y_0,\dots, y_n\},$
implies that
$$
\ker(\partial_1)\cap \dots\cap
\ker(\partial_{n+1})=[\ker(\partial_1),\dots, \ker(\partial_{n+1})]_S
$$
Therefore, there is a natural presentation of the $(n+1)$st
homotopy group of $S^2$ given first by Wu \cite{Wu}:
$$
\pi_{n+1}(S^2)\simeq \frac{\ker(\partial_0)\cap \dots \cap
\ker(\partial_n)}{[\ker(\partial_0),\,\dots\,,\,
\ker(\partial_n)]_S},\ n\geq 1.
$$
For similar results obtained without simplicial constructions see
\cite{EllisMikhailov}.
\para
\begin{theorem}
Let $n\geq 3,$ $F_n$  a free group with a basis $\{x_1,\dots,
x_n\}$. Let $R_i=\langle x_i\rangle^{F_n},\linebreak  i=1,\,\dots\,,\, n,$
$R_{n+1}=\langle x_1\dots x_n\rangle^{F_n}.$ Then
\begin{quote}
$(i)$ there is an isomorphism $Q({\bf r}_1,\dots, {\bf
r}_{n+1})\simeq \mathbb Z;$\para\noindent $(ii)$
$
D(F; ({\bf r}_1\dots {\bf r}_{n+1})_S)=R_1\cap\dots \cap R_{n+1}.
$\end{quote}
Furthermore, $R_1\cap \dots\cap R_{n+1}\neq [R_1,\,\dots\,,\,
R_{n+1}]_S$ for $n\neq 0\mod 8$.
\end{theorem}

\begin{proof}
First apply the functor $\mathbb Z[-]$ to the Milnor construction
$F(S^1)$:

$$\mathbb Z[F(S^1)]:\ \ \ \ldots\ \begin{matrix}\longrightarrow\\[-3.5mm] \ldots\\[-2.5mm]\longrightarrow\\[-3.5mm]
\longleftarrow\\[-3.5mm]\ldots\\[-2.5mm]\longleftarrow \end{matrix}\ \mathbb Z[F_3]\ \begin{matrix}\longrightarrow\\[-3.5mm]\longrightarrow\\[-3.5mm]\longrightarrow\\[-3.5mm]\longrightarrow\\[-3.5mm]\longleftarrow\\[-3.5mm]
\longleftarrow\\[-3.5mm]\longleftarrow
\end{matrix}\ \mathbb Z[F_2]\ \begin{matrix}\longrightarrow\\[-3.5mm] \longrightarrow\\[-3.5mm]\longrightarrow\\[-3.5mm]
\longleftarrow\\[-3.5mm]\longleftarrow \end{matrix}\ \mathbb Z[\mathbb Z]$$
By definition of homology, we have
$$
\pi_n\mathbb Z[F(S^1)]=H_n(\Omega S^2).
$$
From the classical suspension splitting theorem of loop suspensions~\cite{James}, we have
$$
\Sigma \Omega S^2\simeq \bigvee_{k=2}^\infty S^k
$$
and so $H_n(\Omega S^2)=\mathbb{Z}$ for each $n\geq1$. Thus $\pi_n\mathbb Z[F(S^1)]=\mathbb{Z}, \ n\geq1$.
The kernels of homomorphisms
$$
\bar\partial_i:\mathbb Z[F_{n+1}]\to \mathbb Z[F_n],\ i=0,\,\dots\,,\,
n+1
$$
are ideals
$$
(\ker(\partial_i)-1)\mathbb Z[F_{n+1}],\ i=0,\,\dots\,,\, n+1.
$$
Making the enumeration in the free group $F_n$: $x_i=y_{i+1},\
i=0,\,\dots\,,\, n-1,$ lemma \ref{techlemma} (ii) implies that
$$
H_n(\Omega S^2)\simeq Q({\bf r}_1,\,\dots\,,\, {\bf r}_{n+1})\simeq
\mathbb Z
$$
and the statement (i) is proved.
\para
For proving (ii), observe now that there is a natural diagram
$$
\xyma{ \frac{R_1\cap \dots \cap R_{n+1}}{[R_1,\dots,R_{n+1}]_S}
\ar@{=}[d] \ar@{->}[rr]^{\ \ \ \ \ \ \ \ \ \  \ \ \
f_{F;R_1,\dots, R_{n+1}}\ \ \ \ \ \ \ \ \ \ \ \ } & & \frac{{\bf
r}_1\cap
\dots \cap {\bf r}_{n+1}}{({\bf r}_1\dots {\bf r}_{n+1})_S}\ar@{=}[d]\\
\pi_n(\Omega S^2) \ar@{->}[rr] & & H_n(\Omega S^2) }
$$
The homotopy groups $\pi_n(\Omega S^2)=\pi_{n+1}(S^2)$ are finite
for $n\geq 3$, hence the homomorphism $f_{F; R_1,\dots, R_{n+1}}$
is the zero map and therefore,
$$
R_1\cap \dots\cap R_{n+1}\subseteq D(F; ({\bf r}_1\dots {\bf
r}_{n+1})_S).
$$
The reverse inclusion follows trivially, hence the statement (ii)
follows.
\para
Finally, the remark that $R_1\cap \dots\cap R_{n+1}\neq
[R_1,\dots, R_{n+1}]_S$ for $n\neq 0\mod 8$ is just a reformulation of
the result of Curtis \cite{Curtis} that $\pi_n(S^2)\neq 0,\ n\neq
1\mod 8$.
\end{proof}
\para\noindent
{\bf Remark 3.1.}\label{rem11}
For $n=2$, we have the following situation:\para Let $F=F(x_1,x_2),\
R_1=\langle x_1\rangle^F,\ R_2=\langle x_2\rangle^F,\ R_3=\langle
x_1x_2\rangle^F$. Then the following diagram consists of
isomorphisms
$$
\xyma{\mathbb Z\ar@{=}[d] & \mathbb Z \ar@{=}[d]\\ \pi_2(\Omega S^2)\ar@{->}[r]^\simeq \ar@{=}[d] & H_2(\Omega S^2)\ar@{=}[d] \\
\frac{R_1\cap R_2\cap R_3}{[R_1,\,R_2,\,R_3]_S} \ar@{->}[r]^\simeq &
\frac{{\bf r}_1\cap {\bf r}_2\cap {\bf r}_3}{({\bf r}_1{\bf
r}_2{\bf r}_3)_S}}
$$
and therefore $D(F;\,({\bf r}_1{\bf r}_2{\bf
r}_3)_S)=[R_1,\,R_2,\,R_3]_S$.\para
\subsection{Carlsson's Construction}
For any pointed simplicial set $K$ and any group $G$, the Carlsson construction~\cite{Car,Wu1} is the simplicial group $F^G(K)$ in which $F^G(K)_n$ is the self-free product of the group $G$ indexed by non-identity elements in $K_n$ with the face and degeneracy homomorphisms canonically induced from that of $K$. There is a homotopy equivalence~\cite{Car, Wu1}
$
|F^G(K)|\simeq \Omega (BG\wedge |K|),
$
where $BG$ is the classifying space of the group $G$. We consider the case where $K=S^1$ is the simplicial circle and $G$ is an arbitrary group. Note that the suspension of any path-connected space is (weak) homotopy equivalent to $\Sigma BG$ for certain group $G$ by Kan-Thurston's theorem~\cite{KT}. Thus the construction $F^G(S^1)$ gives the loop space model for the suspensions.
\para
The specific information on the simplicial structure of $F^G(S^1)$ is as follows:
\para\noindent The elements in the simplicial circle $S^1$ are listed in~(\ref{1-sphere}). Thus
$$
F^G(S^1)_{n+1}=G_{x_0}\ast G_{x_1}\ast\cdots \ast G_{x_{n}},
$$
where $G_{x_i}$ is a copy of $G$ labeled by $x_i=s_{n}\dots \hat s_i\dots s_0\sigma\in S^1_{n+1}$. The face $\partial_j\colon S^1_{n+1}\to S^1_n$ is given by the formula:
$$
\partial_j x_i=\left\{
\begin{array}{lcl}
x_i&\textrm{ for }& i<j\\
x_{i-1}&\textrm{ for } & i\geq j,\\
\end{array}\right.
$$
where $x_{-1}=x_n=\ast$ in $S^1_n$. Write $g(x_i)$ for the element $g\in G$ located in the copy $G_{x_i}$ of $G$. The group homomorphism
$$
\partial_j\colon F^G(S^1)_{n+1}=G_{x_0}\ast G_{x_1}\ast\cdots \ast G_{x_{n}}\longrightarrow F^G(S^1)_{n}=G_{x_0}\ast G_{x_1}\ast\cdots \ast G_{x_{n-1}}
$$
is given by the formulae:
\begin{equation}\label{equation3.2}
\begin{array}{c}
\partial_0(g(x_i))=\left\{
\begin{array}{lcl}
1&\textrm{ for }& i=0\\
g(x_{i-1})&\textrm{ for }&0<i\leq n,\\
\end{array}\right.\\
\partial_{n+1}(g(x_i))=\left\{
\begin{array}{lcl}
g(x_i)&\textrm{ for }& 0\leq i\leq n-1\\
1&\textrm{ for }& i=n,\\
\end{array}\right.\\
\end{array}
\end{equation}
and for $0<j<n+1$,
\begin{equation}\label{equation3.3}
\partial_{j}(g(x_i))=\left\{
\begin{array}{lcl}
g(x_i)&\textrm{ for }& i<j\\
g(x_{i-1})&\textrm{ for }& i\geq j.\\
\end{array}\right.\\
\end{equation}

In the free product $G^{\ast n+1}=G_{x_0}\ast G_{x_1}\ast\cdots
\ast G_{x_{n}}$, let $R^G_{n+1,\,0}=\langle g(x_0) \ | \ g\in
G\rangle^{G^{\ast n+1}}$, $R_{n+1,\,n+1}=\langle g(x_n) \ | \ g\in
G\rangle^{G^{\ast n+1}}$ and $R^G_{n+1,\,j}=\langle
g(x_{j-1})^{-1}g(x_j) \ | \ g\in G\rangle^{G^{\ast n+1}}$ for
$0<j<n+1$. Let $\br^G_{n+1,\,j}=(R^G_{n+1,\,j}-1)\Z[G^{\ast n+1}]$.

\para
\begin{theorem}\label{theorem3.2}
Let $G$ be any group. Then there is an isomorphism of groups
$$
Q(\br^G_{n+1,\,0},\br^G_{n+1,\,1},\,\ldots\,,\,\br^G_{n+1,\,n+1})\cong H_{n+1}(\Omega\Sigma BG;\,\Z)\cong \bigoplus_{k=1}^\infty H_{n+1}((BG)^{\wedge k};\,\Z),
$$
where $X^{\wedge k}$ is the $k$-fold self smash product of $X$.
\end{theorem}

\begin{proof}
Let $T=\{t_{\alpha}\ | \ \alpha\in J\}$ be a set of generators for $G$ and let $F$ be the free group generated by $T$. Consider Carlsson's construction $F^{F}(S^1)$. The group $F^F(S^1)_{n+1}$ is the free group generated by $\{t_{\alpha}(x_j) \ | \ \alpha\in J, \ 0\leq j\leq n\}$.
Let $y^{(\alpha)}_j=t_{\alpha}(x_j)t_{\alpha}(x_{j+1})^{-1}$ for $-1\leq j\leq n$ with $t_{\alpha}(x_{-1})=t_{\alpha}(x_{n+1})=1$. Then
$$
\{y^{(\alpha)}_j \ | \ \alpha\in J,\ 0\leq j\leq n\}
$$
is also a basis for $F^F(S^1)_{n+1}$. From formulae~(\ref{equation3.2}) and~(\ref{equation3.3}),
$$
\partial_j(y^{(\alpha)}_k)=\left\{
\begin{array}{lcl}
y^{(\alpha)}_{k-1}&\textrm{ for }& j\leq k\\
1&\textrm{ for }& j=k+1\\
y^{\alpha}_k&\textrm{ for }& j>k+1.\\
\end{array}\right.
$$
Thus
$$
\begin{array}{rcl}
\ker(\partial_j\colon F^F(S^1)_{n+1}\to F^F(S^1)_{n})&=&\langle y^{(\alpha)}_{j-1}\ | \ \alpha\in J\rangle^{F^F(S^1)_{n+1}}\\
&=&\langle t_{\alpha}(x_{j-1})t_{\alpha}(x_{j})^{-1}\ | \ \alpha\in J\rangle^{F^F(S^1)_{n+1}}\\
&=&R^F_{n+1,\,j}.\\
\end{array}
$$
for $0\leq j\leq n+1$.
The canonical epimorphism $\phi\colon F\to G$ induces a simplicial epimorphism
$$\tilde \phi\colon F^F(S^1)\twoheadrightarrow F^G(S^1),$$
which induces the epimorphism
$$
\tilde \phi|\colon \ker(\partial_j\colon F^F(S^1)_{n+1}\to F^F(S^1)_{n})\twoheadrightarrow \ker(\partial_j\colon F^G(S^1)_{n+1}\to F^G(S^1)_{n}).
$$
Thus
$$
\begin{array}{rcl}
\ker(\partial_j\colon F^G(S^1)_{n+1}\to F^G(S^1)_{n})&=&\tilde\phi(R^F_{n+1,\,j})\\
&=&R^G_{n+1,\,j}\\
\end{array}
$$
for $0\leq j\leq n+1$. Note that the faces
$$
\partial_1,\,\ldots\,,\,\partial_n\colon F^F(S^1)_{n+1}\longrightarrow F^F(S^1)_n
$$
are natural projections under the basis $\{y^{(\alpha)}_j \ | \ \alpha\in J,\ 0\leq j\leq n\}$. By Lemma~\ref{techlemma}, the Moore chains of the
 simplicial group $\Z[F^F(S^1)]$
$$
N_{n+2}(\Z[F^F(S^1)])=\br^F_{n+2,1}\cap
\br^F_{n+2,\,2}\cap\cdots\cap \br^F_{n+2,\,n+2}=(\br^F_{n+2,\,1}
\br^F_{n+2,\,2}\,\cdots\,\br^F_{n+2,\,n+2})_S.
$$
Thus the Moore boundary
$$
\begin{array}{rcl}
\mathcal{B}_{n+1}(\Z[F^F(S^1)])&=&\partial_0(N_{n+2}(\Z[F^F(S^1)]))\\
&=&\partial_0((\br^F_{n+2,\,1} \br^F_{n+2,\,2}\cdots\br^F_{n+2,\,n+2})_S)\\
&=&(\br^F_{n+1,\,0} \br^F_{n+1,\,1}\,\cdots\,\br^F_{n+1,\,n+1})_S.\\
\end{array}
$$
Now the simplicial epimorphism $\tilde\phi\colon F^F(S^1)\twoheadrightarrow F^G(S^1)$ extends canonically to a simplicial epimorphism
$$
\Z[\tilde\phi]\colon \Z[F^F(S^1)]\twoheadrightarrow \Z[F^G(S^1)],
$$
which induces an epimorphism on the Moore boundaries
$$
\Z(\tilde\phi)|\colon
\mathcal{B}_{n+1}(\Z[F^F(S^1)])\twoheadrightarrow
\mathcal{B}_{n+1}(\Z[F^G(S^1)]).
$$
Thus
$$
\begin{array}{rcl}
\mathcal{B}_{n+1}(\Z(F^G(S^1)))&=&\Z[\tilde\phi]((\br^F_{n+1,\,0} \br^F_{n+1,\,1}\,\cdots\,\br^F_{n+1,\,n+1})_S)\\
&=&(\br^G_{n+1,\,0} \br^G_{n+1,\,1}\,\cdots\,\br^G_{n+1,\,n+1})_S.\\
\end{array}
$$
Note that the Moore cycles
$$
\begin{array}{rcl}
\mathcal{Z}_{n+1}(F^G(S^1))&=&\bigcap\limits_{j=0}^n\ker(\partial_j\colon \Z[F^G(S^1)]_{n+1}\to\Z[F^G(S^1)]_n)\\
&=&\br^G_{n+1,\,0}\cap \br^G_{n+1,\,1}\cap\cdots\cap \br^G_{n+1,\,n+1}.\\
\end{array}
$$
It follows that
$$
\begin{array}{rcl}
Q(\br^G_{n+1,\,0},\,\br^G_{n+1,\,1},\,\ldots\,,\,\br^G_{n+1,\,n+1})&=&\mathcal{Z}_{n+1}(F^G(S^1))/\mathcal{B}_{n+1}(F^G(S^1))\\
&=&\pi_{n+1}(\Z(F^G(S^1)))\\
&\cong&H_{n+1}(F^G(S^1);\Z)\\
&\cong&H_{n+1}(\Omega \Sigma BG;\Z)\\
&\cong&\bigoplus\limits_{k=1}^\infty H_{n+1}((BG)^{\wedge k};\,\Z),\\
\end{array}
$$
where the last isomorphism follows from the classical suspension splitting theorem of loop suspensions~\cite{James}, hence the assertion.
\end{proof}
\para
\begin{cor}
Let $G$ be a group. Then
$$
Q(\br^G_{n+1,\,0},\,\br^G_{n+1,1},\,\ldots\,,\,\br^G_{n+1,n+1})=0
$$
for all $n\geq 0$ if and only if the reduced homology $\tilde H_*(G;\,\Z)=0$.
\end{cor}
\para
\begin{cor}
Let $G$ be a group. Then the groups
$$
Q(\br^G_{n+1,\,0},\,\br^G_{n+1,\,1},\,\ldots\,,\,\br^G_{n+1,n+1})
$$
is torsion-free for all $n\geq 0$ if and only if the integral
homology $\tilde H_*(G;\,\Z)$ is torsion-free.
\end{cor}
\para
\begin{example}
{\rm The group $F=F^{\Z/2}(S^1)_2=\Z/2\ast\Z/2$ is generated by
$x_0,\,x_1$ with defining relations $x_0^2=x_1^2=1$. In this case,
\begin{align*}
& \br^{\Z/2}_{2,\,0}=(\langle x_0\rangle^F-1)\Z[\Z/2\ast\Z/2],\\\ &
\br^{\Z/2}_{2,\,1}=(\langle x_0x_1\rangle^F-1)\Z[\Z/2\ast\Z/2],\\\ &
\br^{\Z/2}_{2,\,2}=(\langle x_1\rangle^F-1)\Z[\Z/2\ast\Z/2]
\end{align*}
with
$$
\begin{array}{rcl}
Q(\br^{\Z/2}_{2,\,0}, \br^{\Z/2}_{2,\,1},\br^{\Z/2}_{2,\,2})&\cong& H_2(\mathbb{R}\mathrm{P}^\infty;\,\Z)\oplus H_2(\mathbb{R}\mathrm{P}^\infty\wedge\mathbb{R}\mathrm{P}^\infty;\,\Z)\\
&=&\Z/2.\quad\Box
\end{array}
$$}
\end{example}
\para\noindent
{\bf Remark 3.2.}
Observe that, for every group $G$, there is the following natural
diagram (see \cite{BL}):
$$
\xyma{H_3(G) \ar@{->}[d] \ar@{->}[r] & \Gamma_2(G_{ab})\ar@{=}[d]
\ar@{->}[r] & \pi_2(\Omega\Sigma BG) \ar@{->>}[r] \ar@{->}[d]& H_2(G)\ar@{->}[d]\\
H_3(G_{ab}) \ar@{->}[r] & \Gamma_2(G_{ab})\ar@{->}[r] &
G_{ab}\otimes G_{ab} \ar@{->>}[r] & H_2(G_{ab}),}
$$
where $\Gamma_2$ is the universal Whitehead quadratic functor. For
a group $G$ with $\ker\{H_2(G)\to H_2(G_{ab})\}=0$ and
torsion-free $G_{ab}$, the natural map $\pi_2(\Omega\Sigma BG)\to
G_{ab}\otimes G_{ab}(\subseteq H_2(\Omega\Sigma BG))$ is a
monomorphism. In particular, this covers the case mentioned in
Remark \ref{rem11}.

\par\vspace{1cm}\noindent
Roman Mikhailov\\Steklov Mathematical Institute\\
Department of Algebra\\
Gubkina 8\\
Moscow 119991
Russia\\
email: romanvm@mi.ras.ru
\par\vspace{.5cm}\noindent
Inder Bir S. Passi\\
Centre for Advanced Study in Mathematics\\
Panjab University \\
Chandigarh 160014 India
\par\noindent
and
\par\noindent
Indian Institute of Science Education and Research  Mohali\\
MGSIPA Complex, Sector 19\\
Chandigarh 160019
India\\
email: ibspassi@yahoo.co.in\par\vspace{.5cm}\noindent
Jie Wu\\
Department of Mathematics\\
National University of Singapore\\
Singapore\\
email: matwuj@nus.edu.sg

\end{document}